\documentclass[11pt,reqno]{amsart}

\numberwithin{equation}{section}

\numberwithin{equation}{section}
\usepackage[all,cmtip]{xy}
\usepackage{amssymb,rotating,amsmath,amsfonts, 
amsthm,color,graphics,epsfig,bbm,bm,manfnt,psfrag}

\newcommand{\calD}{\mathcal{D}}

\newcommand{\calS}{\mathcal{S}}

\newcommand{\mC}{\mathbb{C}}

\newcommand{\mN}{\mathbb{N}}

\newcommand{\mR}{\mathbb{R}}

\newcommand{\mZ}{\mathbb{Z}}

\newcommand{\bbm}{\mathbf{m}}

\newcommand{\bbu}{\mathbf{u}}
\newcommand{\bbv}{\mathbf{v}}

\newcommand{\bbx}{\mathbf{x}}

\newcommand{\bbS}{\mathbf{S}}

\newcommand{\nm}{\,\rule[-.6ex]{.13em}{2.3ex}\,}

\newtheorem{theorem}{Theorem}[section]

\theoremstyle{definition}

\theoremstyle{definition}
\newtheorem{definition}[theorem]{Definition}

\theoremstyle{definition}
\newtheorem{notation}[theorem]{Notation}

\begin{document}

\title[Spatially invariant behaviours]
      {Algebraic characterization of approximate controllability  
       of behaviours of spatially invariant systems}
       
\author{Amol Sasane}
\address{Department of Mathematics, Faculty of Science, 
    Lund University, Sweden.}
\email{amol.sasane@math.lu.se}


\keywords{systems of linear partial differential equations with
  constant coefficients, approximate controllability, controllability,
  Fourier transformation, behaviours, distributions that are periodic
  in the spatial directions}

\begin{abstract}
An algebraic characterization of the property of approximate
controllability is given, for behaviours of spatially invariant dynamical
systems, consisting of distributional solutions $w$, that are periodic
in the spatial variables, to a system of partial differential
equations
$$
M\left(
\frac{\partial}{\partial x_1}, \cdots, \frac{\partial}{\partial x_d}, 
\frac{\partial}{\partial t}
\right) w=0,
$$
corresponding to a polynomial matrix $M\in ({\mathbb{C}}
[\xi_1,\dots, \xi_d, \tau])^{m\times n}$. This settles an issue 
left open in \cite{Sas}.
\end{abstract}

\maketitle

\section{Introduction}

Consider a homogeneous, linear, constant coefficient partial 
differential equation, in ${\mathbb{R}}^{d+1}$ described by a
polynomial  $p\in {\mathbb{C}}[\xi_1,\dots, \xi_d, \tau]$:
\begin{equation}
\label{eq_ker_rep_p}
p\left(
\frac{\partial}{\partial x_1}, \cdots, \frac{\partial}{\partial x_d}, 
\frac{\partial}{\partial t}
\right) w=0.
\end{equation}
That is, the differential operator 
$$
p\left( \displaystyle 
\frac{\partial}{\partial x_1}, \cdots, \frac{\partial}{\partial x_d}, 
\frac{\partial}{\partial t} \right)
$$ 
is obtained from the polynomial $p\in {\mathbb{C}}[\xi_1,\dots,\xi_d, \tau]$ 
by making the replacements
$$
\xi_k\rightsquigarrow
\frac{\partial}{\partial x_k} \;\textrm{ for } k=1, \dots, d, 
\;\textrm{ and }\;\;
\tau\rightsquigarrow \frac{\partial}{\partial t}.
$$ 
More generally, given a polynomial {\em matrix} $M \in
({\mathbb{C}}[\xi_1,\dots, \xi_d, \tau])^{m\times n}$, consider the
corresponding {\em system} of partial differential equations
\begin{equation}
\label{eqn_PDE_system}
M\left( \displaystyle 
\frac{\partial}{\partial x_1}, \cdots, \frac{\partial}{\partial x_d}, 
\frac{\partial}{\partial t}\right) w:=
\left[\begin{array}{ccc} 
\displaystyle \sum_{j=1}^n p_{1j} \left( \displaystyle 
\frac{\partial}{\partial x_1}, \cdots, \frac{\partial}{\partial x_d}, 
\frac{\partial}{\partial t}\right) w_{j} 
\\ \vdots \\ 
\displaystyle 
\sum_{j=1}^n p_{mj} \left( \displaystyle 
\frac{\partial}{\partial x_1}, \cdots, \frac{\partial}{\partial x_d}, 
\frac{\partial}{\partial t}\right) w_{j}
\end{array} \right] = 0,
\end{equation}
where solutions $w$ now have the $n$ components $w_1, \dots, w_n$, and
$M=[p_{ij}]$ with $p_{ij}$ denoting the polynomial entries of $M$ for
$1\leq i\leq m$ and $1\leq j \leq n$.

In the behavioural approach to control theory pioneered by Willems
\cite{PolWil}, the ``behaviour'' ${\mathfrak{B}}_{{\mathcal{W}}}(M)$
associated with $M$ in ${\mathcal{W}}^n$ (where ${\mathcal{W}}$ is an
appropriate solution space, for example smooth functions
$C^\infty({\mathbb{R}}^{d+1})$ or distribution spaces like
${\mathcal{D}}'({\mathbb{R}}^{d+1})$ or ${\mathcal{S}}'({\mathbb{R}}^{d+1})$ 
and so on), is defined to be the set of all solutions 
$w\in {\mathcal{W}}^n$ that satisfy the above partial differential equation 
system \eqref{eqn_PDE_system}. Let us recall the notion of a behaviour 
associated with a system of partial differential equations associated 
with a polynomial matrix $M$. 

\begin{definition}[Solution space invariant under differentiation; Behaviour]
Let ${\mathcal{W}}$ be a subspace of $({\mathcal{D}}'({\mathbb{R}}^{d+1}))^n$ 
which is {\em invariant under differentiation}, that is, for all $w\in \mathcal{W}$, 
\begin{eqnarray*}
&& \frac{\partial}{\partial x_k} w\in \mathcal{W}, \textrm{ for all } k=1,\cdots, d, \textrm{ and}\\
&& \frac{\partial}{\partial t} w\in \mathcal{W}.
\end{eqnarray*}
The {\em behaviour ${\mathfrak{B}}_{{\mathcal{W}}}(M)$
associated with $M\in
({\mathbb{C}}[\xi_1,\dots, \xi_d, \tau])^{m\times n}$ in ${\mathcal{W}}^n$} is 
$$
{\mathfrak{B}}_{{\mathcal{W}}}(M):=
\left\{ w\in {\mathcal{W}}^n: M\left( \displaystyle 
\frac{\partial}{\partial x_1}, \cdots, \frac{\partial}{\partial x_d}, 
\frac{\partial}{\partial t}\right) w=0\right\}.
$$
\end{definition}

The aim in the behavioural approach to control theory is then to obtain 
algebraic characterizations (in terms of algebraic properties 
of the polynomial matrix $M$) of certain analytical properties of 
${\mathfrak{B}}_{{\mathcal{W}}}(M)$ (for example, the control theoretic 
properties of autonomy, controllability, stability, and so on). 
We refer the reader to \cite{PolWil} for background on the 
behavioural approach in the case of systems of ordinary differential
equations, and to \cite{BalSta}, \cite{PilSha}, \cite{SasThoWil} for
distinct takes on this in the context of systems described by partial
differential equations.

The goal of this article is to give algebraic characterizations of the
properties of approximate controllability of behaviours of spatially
invariant dynamical systems, consisting of distributional solutions
$w$, that are periodic in the spatial variables, to a system of
partial differential equations
$$
M\left(
\frac{\partial}{\partial x_1}, \cdots, \frac{\partial}{\partial x_d}, 
\frac{\partial}{\partial t}
\right) w=0,
$$ 
corresponding to a polynomial matrix 
$M\in ({\mathbb{C}}[\xi_1,\dots, \xi_d, \tau])^{m\times n}$. 
This settles a question left open in \cite{Sas}. 

We remark that there has been recent
interest in ``spatially invariant systems'', see for example
\cite{Cur},  \cite{CurSas}, where one considers solutions to
partial differential equations that are periodic along the spatial
direction.

We give the relevant definitions below, and also state our main
results in Theorem~\ref{main_theorem} (characterizing approximate
controllability).

\subsection{Controllability and approximate controllability} 

Let us first recall the property of ``controllability'', which means
the following.

\begin{definition}[Controllability; Approximate controllability]
Let ${\mathcal{W}}$ be a subspace of $({\mathcal{D}}'({\mathbb{R}}^{d+1}))^n$ 
which is invariant under differentiation, and suppose that $M\in ({\mathbb{C}}[\xi_1, \dots,
  \xi_d,\tau])^{m\times n}$. 
\begin{enumerate}
 \item 
The  behaviour
${\mathfrak{B}}_{{\mathcal{W}}}(M)$ in ${\mathcal{W}}$ is called {\em
  controllable in time} $\textrm{T}> 0$ if for every $w_1, w_2 \in
{\mathfrak{B}}_{{\mathcal{W}}}(M)$, there is a $w\in
{\mathfrak{B}}_{{\mathcal{W}}}(M)$ such that 
\begin{eqnarray*}
 w|_{(-\infty, 0)}&=&w_1|_{(-\infty,0)}\textrm{  and}\\
 w|_{(\textrm{T} , +\infty)}&=&w_2|_{(\textrm{T} , +\infty)}
\end{eqnarray*}

\item  
The behaviour ${\mathfrak{B}}_{{\mathcal{W}}}(M)$
in ${\mathcal{W}}$ is called {\em approximately controllable in time}
$\textrm{T}> 0$ if for every $\epsilon>0$ and for all $w_1, w_2 \in
{\mathfrak{B}}_{{\mathcal{W}}}(M)$, there is a $w\in
{\mathfrak{B}}_{{\mathcal{W}}}(M)$ such that 
\begin{eqnarray*}
 w|_{(-\infty, 0)}&=&w_1|_{(-\infty,0)},
\end{eqnarray*}
and  $(w-w_2)|_{(\textrm{T} , +\infty)}$ is a
regular distribution on $(\textrm{T} , +\infty)\times \mR^d$ with $$
\sup_{(t,\bbx)\in (\textrm{T} , +\infty)\times \mR^d} \nm
(w-w_2)|_{(\textrm{T} , +\infty)} ( t,\bbx)\nm_{2} <\epsilon.
 $$
 \end{enumerate}
 See Figure~\ref{c_versus_ac}.
 \begin{figure}[h]
    \center
    \psfrag{p}[c][c]{$w_1$}
    \psfrag{f}[c][c]{$w_2$}
    \psfrag{e}[c][c]{$\epsilon$}
    \psfrag{T}[c][c]{$\mathrm{T}$}
    \psfrag{w}[c][c]{$w$}
    \psfrag{0}[c][c]{$0$}
    \includegraphics[width=12.9 cm]{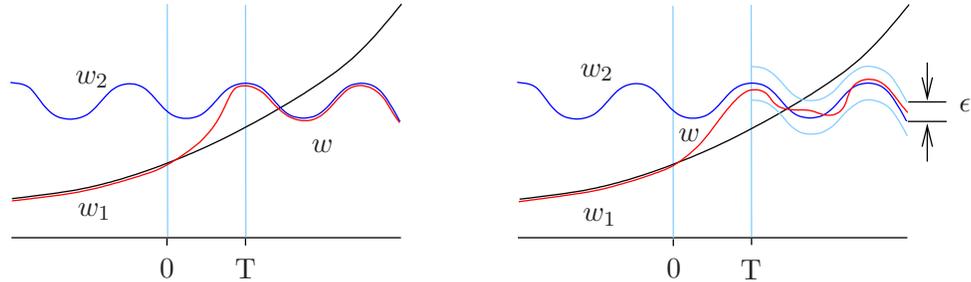}
    \caption{Controllability versus approximate controllability.}
    \label{c_versus_ac}
 \end{figure}
\end{definition}

Our main result is the following. 

\begin{theorem}
\label{main_theorem}
Suppose that ${\mathbb{A}}:=\{{\mathbf{a_1}}, \dots, {\mathbf{a_d}}\}$
is a linearly independent set of vectors in ${\mathbb{R}}^d$. Let $M \in
({\mathbb{C}}[\xi_1,\dots, \xi_d, \tau])^{m\times n}$ and let
$$
{\mathfrak{B}}_{{\mathcal{D}}'_{{\mathbb{A}}}({\mathbb{R}}^{d+1})}(M):=
\left\{w\in ({\mathcal{D}}'_{{\mathbb{A}}}({\mathbb{R}}^{d+1}))^{n}: 
M\left( \frac{\partial}{\partial x_1},\cdots, \frac{\partial}{\partial x_d}, 
\frac{\partial}{\partial t}\right) w=0\right\}.
$$ 
Then the following statements are equivalent:
\begin{enumerate}
 \item
   ${\mathfrak{B}}_{{\mathcal{D}}'_{{\mathbb{A}}}({\mathbb{R}}^{d+1})}(M)$
   is approximately controllable in time $\textrm{\em T}>0$.
 \item
   ${\mathfrak{B}}_{{\mathcal{D}}'_{{\mathbb{A}}}({\mathbb{R}}^{d+1})}(M)$
   is controllable in time $\textrm{\em T}>0$.
 \item For each $\mathbf{v}\in A^{-1}{\mathbb{Z}}^d$, there exists an
   $r_{\mathbf{v}}\in \{0,1,2,3,\cdots\}$ satisfying $ r_{\mathbf{v}}
   \leq \min\{n,m\}$ and such that for all $t\in {\mathbb{C}}$, $
   \textrm{\em rank}\left( M( 2\pi i \mathbf{v}, t) \right) =
   r_{\mathbf{v}}. $
\item For each $\mathbf{v}\in A^{-1}{\mathbb{Z}}^d$, the
  $\mC[\tau]$-module $\mC[\tau]^{1\times n}/\langle M(2\pi i \bbv,
  \tau) \rangle$ is torsion free.
\end{enumerate}
\end{theorem}

Here ${\mathcal{D}}'_{{\mathbb{A}}}({\mathbb{R}}^{d+1})$ is, roughly
speaking, the set of all distributions on ${\mathbb{R}}^{d+1}$ that
are periodic in the spatial direction with a discrete set
${\mathbb{A}}$ of periods. The precise definition of $
{\mathcal{D}}'_{{\mathbb{A}}}({\mathbb{R}}^{d+1})$ is given below in
Subsection~\ref{subsection_sol_space}. 

The algebraic terminology in (4) of Theorem~\ref{main_theorem} is explained below. 
Consider the polynomial matrix 
\[ 
M=\left[ 
\begin{array}{ccc} 
p_{11} & \dots & p_{1n}\\ 
\vdots & & \vdots \\ 
p_{m1} & \dots & p_{mn} 
\end{array} 
\right] \in \mC[\tau]^{m \times n}.
\] 
Then each row of $M$ is an element of the free
$\mC[\tau]$-module $\mC[\tau]^{1\times n}$.

\begin{notation}[$\langle M\rangle$]
 Given $M\in \mC[\tau]^{m \times n}$, 
let $\langle M\rangle$ denote the $\mC[\tau]$-submodule
of $\mC[\tau]^{1\times n}$ generated by the rows of the polynomial
matrix $M$. 
\end{notation}

\begin{definition}[Torsion element; Torsion free module]
Let $M\in \mC[\tau]^{m \times n}$. 
 \begin{enumerate}
  \item An element $[\bbm] $ in the quotient
$\mC[\tau]$-module $\mC[\tau]^{1\times n}/ \langle M( \tau) \rangle$
(corresponding to an element $\bbm \in \langle R \rangle$) is called a {\em
  torsion element} if there exists a polynomial $p\in \mC[\tau]$ such
that $p \cdot[\bbm] =[\mathbf{0}]$, that is, $p\cdot \bbm \in \langle
M \rangle$.
\item The quotient $\mC[\tau]$-module $\mC[\tau]^{1\times n}/
\langle M( \tau) \rangle$ is said to be {\em torsion free} if it has
no nontrivial torsion element.
 \end{enumerate}

\end{definition}

The equivalence of (2) and (3) follows from the proof of
\cite[Theorem~1.4]{Sas}.

\subsection{The space ${\mathcal{D}}'_{{\mathbb{A}}}({\mathbb{R}}^{d+1})$} 
\label{subsection_sol_space}

\begin{definition}[Translation operator ${\mathbf{S_a}}$; Periodic distribution] 

$\;$

\noindent Let ${\mathbf{a}}\in {{\mathbb{R}}}^{d}$.
\begin{enumerate}
 \item The {\em translation
   operation} ${\mathbf{S_a}}$ on distributions in
 ${\mathcal{D}}'({\mathbb{R}}^d)$ is defined by
 $
\langle {\mathbf{S_a}}(T),\varphi\rangle=\langle
 T,\varphi(\cdot+{\mathbf{a}})\rangle$ for all $\varphi \in
 {\mathcal{D}}({\mathbb{R}}^d).$ 
\item A distribution $T\in {\mathcal{D}}'({\mathbb{R}}^d)$ is said to be
 {\em periodic with a period} $\mathbf{a}\in {\mathbb{R}}^d$ if $T=
 {\mathbf{S_a}}(T)$.
\end{enumerate}
\end{definition}

\begin{notation}[${\mathbb{A}}$, $A$, ${\mathcal{D}}'_{{\mathbb{A}}}({\mathbb{R}}^d)$] 

$\;$ 

\noindent Let ${\mathbb{A}}:=\{{\mathbf{a_1}}, \dots, {\mathbf{a_d}}\}$ be a
linearly independent set vectors in ${\mathbb{R}}^d$.  It will be convenient for the sequel 
to also introduce the following matrix: 
\begin{equation} 
\label{equation_matrix_A}
A:= \left[ \begin{array}{ccc} 
    \mathbf{a_1}^{\top} \\ \vdots \\ \mathbf{a_d}^{\top} 
    \end{array}\right].
\end{equation}

\noindent ${\mathcal{D}}'_{{\mathbb{A}}}({\mathbb{R}}^d)$ is the set of all
distributions $T\in \calD'(\mR^d)$ that satisfy 
$$
{\mathbf{S_{a_k}}}(T)=T\textrm{ for all }k=1,\dots, d.
$$
\end{notation}

From \cite[\S34]{Don}, $T\in {\mathcal{D}}'_{{\mathbb{A}}}({\mathbb{R}}^d)$ is a tempered distribution, and from the
above it follows by taking Fourier transforms that $ (1-e^{2\pi i
  {\mathbf{a_k}} \cdot \mathbf{y}})\widehat{T}=0$ for $ k=1,\dots,d.
$ It can be seen that
$$
\widehat{T}=\sum_{\mathbf{v} \in A^{-1} {\mathbb{Z}}^d} \alpha_{\mathbf{v}}(T) \delta_{\mathbf{v}},
$$ 
for some scalars $\alpha_{\mathbf{v}}\in {\mathbb{C}}$, and where
$A$ is the matrix given in \eqref{equation_matrix_A}. 
Also, in the above, $\delta_{\mathbf{v}}$ denotes the usual Dirac
measure with support in $\mathbf{v}$:
$$ 
\langle \delta_{\mathbf{v}}, \psi\rangle =\psi (\mathbf{v})
\;\textrm{ for }\psi \in {\mathcal{D}}'({\mathbb{R}}^d).
$$ 
By the Schwartz Kernel Theorem (see for instance \cite[p.~128,
  Theorem~5.2.1]{Hor}), ${\mathcal{D}}'({\mathbb{R}}^{d+1})$ is
isomorphic as a topological space to
${\mathcal{L}}({\mathcal{D}}({\mathbb{R}}),
{\mathcal{D}}'({\mathbb{R}}^d))$, the space of all continuous linear
maps from ${\mathcal{D}}({\mathbb{R}})$ to
${\mathcal{D}}'({\mathbb{R}}^d)$, thought of as vector-valued
distributions. For preliminaries on vector-valued distributions, we
refer the reader to \cite{Car}. We indicate this isomorphism by
putting an arrow on top of elements of
${\mathcal{D}}'({\mathbb{R}}^{d+1})$. 

\begin{notation}[$ 
{\mathcal{D}}'_{{\mathbb{A}}}({\mathbb{R}}^{d+1})$] 
 
 $\;$
 
 \noindent For $w\in
{\mathcal{D}}'({\mathbb{R}}^{d+1})$, we set $\vec{w}\in
{\mathcal{L}}({\mathcal{D}}({\mathbb{R}}),
{\mathcal{D}}'({\mathbb{R}}^d))$ to be the vector-valued distribution
defined by
 $
\langle \vec{w}(\varphi), \psi\rangle=\langle w, \psi \otimes \varphi\rangle 
$  
for $\varphi\in {\mathcal{D}}({\mathbb{R}})$ and $\psi \in
        {\mathcal{D}}({\mathbb{R}}^d)$.
        
        \medskip 
        
\noindent If ${\mathbb{A}}:=\{{\mathbf{a_1}}, \dots, {\mathbf{a_d}}\}$ is a
linearly independent set vectors in ${\mathbb{R}}^d$, then we define
$$ 
{\mathcal{D}}'_{{\mathbb{A}}}({\mathbb{R}}^{d+1}):=\{w\in
        {\mathcal{D}}'({\mathbb{R}}^{d+1}): \textrm{ for all }\varphi
        \in {\mathcal{D}}({\mathbb{R}}), \; \vec{w}(\varphi) \in
            {\mathcal{D}}'_{{\mathbb{A}}}({\mathbb{R}}^d)\}.
$$        
\end{notation}

\bigskip 

\noindent For $w\in {\mathcal{D}}'_{{\mathbb{A}}}({\mathbb{R}}^{d+1})$,
$$ 
\frac{\partial}{\partial x_k}w \in
{\mathcal{D}}'_{{\mathbb{A}}}({\mathbb{R}}^{d+1}) \;\textrm{ for }
k=1, \dots, d, \; \textrm{ and } \;\; \frac{\partial}{\partial t}w \in
{\mathcal{D}}'_{{\mathbb{A}}}({\mathbb{R}}^{d+1}).
$$ 
Also, for $w\in {\mathcal{D}}'_{{\mathbb{A}}}({\mathbb{R}}^{d+1})$,
we define $\widehat{w}\in {\mathcal{D}}'({\mathbb{R}}^{d+1})$ by
\begin{equation} \label{equation_spatial_fourier_transform}
\langle \widehat{w}, \psi \otimes \varphi \rangle
=
\langle \vec{w}(\varphi), \widehat{\psi} \rangle,
\end{equation}
for $\varphi \in {\mathcal{D}}({\mathbb{R}})$ and $\psi \in
        {\mathcal{D}}({\mathbb{R}}^d)$. In the right hand side 
        of \eqref{equation_spatial_fourier_transform},  
        $\widehat{\cdot}$ is the usual Fourier transform 
        $\psi \mapsto \widehat{\psi}: \calS(\mR^d) \rightarrow \calS(\mR^d)$ 
        on the Schwartz space $\calS(\mR^d)$ of test functions 
        with rapidly decreasing derivatives.
        That \eqref{equation_spatial_fourier_transform} specifies a
        well-defined distribution in
        ${\mathcal{D}}'({\mathbb{R}}^{d+1})$, can be seen using the
        fact that for every $\Phi\in
        {\mathcal{D}}({\mathbb{R}}^{d+1})$, there exists a sequence of
        functions $(\Psi_n)_n$ that are finite sums of direct products
        of test functions, that is, $\Psi_n=\sum_{k} \psi_{k} \otimes
        \varphi_k$, where $\psi_k\in {\mathcal{D}}({\mathbb{R}}^d)$
        and $\varphi_k\in {\mathcal{D}}({\mathbb{R}})$, such that
        $\Psi_n$ converges to $\Phi$ in
        ${\mathcal{D}}({\mathbb{R}}^{d+1})$. We also have
$$ 
\widehat{ \frac{\partial}{\partial x_k} w}= 2\pi i y_k \widehat{w}
        \;\textrm{ for } k=1, \dots, d,\; \textrm{ and }\;\; \widehat{
          \frac{\partial}{\partial t} w}=\frac{\partial }{\partial
          t}\widehat{w}.
$$
Here $\mathbf{y}=(y_1,\dots, y_d)$ is  the Fourier transform variable.

\section{Proof of Theorem~\ref{main_theorem}}

Before we prove our main result, we illustrate the key idea behind the
proof of our algebraic condition. For a trajectory in the behaviour,
by taking Fourier transform with respect to the spatial variables, the
partial derivatives with respect to the spatial variables are
converted into the polynomial coefficients $c_{ij}(2\pi i \mathbf{y}
)$, where $\mathbf{y}$ is the vector of Fourier transform variables
$y_1 ,\dots, y_d$. But the support of $\widehat{w}$ is carried on a
family of lines, indexed by $\mathbf{n}\in{\mathbb{Z}}^d$, in
${\mathbb{R}}^{d+1}$, parallel to the time axis. So we obtain a family
of ordinary differential equations, parameterized by
$\mathbf{n}\in{\mathbb{Z}}^d$, and by ``freezing'' an $\mathbf{n}\in
{\mathbb{Z}}^d$, we get an ordinary differential equation. So
essentially the proof is completed by looking at the ordinary
differential equation characterizations of controllability and
approximate controllability, and it turns out that the two notions
actually coincide there.

\begin{proof}[Proof of Theorem~\ref{main_theorem}]  
We will show that (1)$\Rightarrow$(4)$\Rightarrow$(3)$\Rightarrow$(2)$\Rightarrow$(1). 

\medskip 

\noindent {\bf (1) $\Rightarrow$ (4)}: Suppose that $ {\mathfrak{B}}_{
  {\mathcal{D}}'_{{\mathbb{A}}}({\mathbb{R}}^{d+1}) }(M)$ is
approximately controllable. Let $\mathbf{v}\in
A^{-1}{\mathbb{Z}}^d$. Suppose that $\Theta\in
{\mathfrak{B}}_{{\mathcal{D}}'({\mathbb{R}})} (M(2\pi i {\mathbf{v}},
\tau))$. Set
\begin{eqnarray*}
 w_1&:=&0,\\
 w_2&:=& e^{2\pi i {\mathbf{v}}\cdot \mathbf{x}}  \otimes \Theta. 
\end{eqnarray*}
Then $w_1,w_2 \in ({\mathcal{D}}'_{{\mathbb{A}}}({\mathbb{R}}^{d+1}))^{n}$,  
since  for all $k\in \{1,\dots, d\}$, we have 
$$
\mathbf{S_{a_k}} w_2= e^{2\pi i \mathbf{v}\cdot (\mathbf{x}+\mathbf{a_k})}\otimes \Theta
=e^{2\pi i \mathbf{v}\cdot \mathbf{a_k}}e^{2\pi i \mathbf{v}\cdot \mathbf{x}}\otimes \Theta
= 1\cdot e^{2\pi i \mathbf{v}\cdot \mathbf{x}}\otimes \Theta= w_2.
$$ 
Also, $w_1,w_2\in
{\mathfrak{B}}_{{\mathcal{D}}'_{{\mathbb{A}}}({\mathbb{R}}^{d+1})}(M)$,
because
$$ 
M\left(\frac{\partial}{\partial x_1},\cdots,
\frac{\partial}{\partial x_d} , \frac{\partial}{\partial t}\right) w_2
= e^{2\pi i \mathbf{v}\cdot \mathbf{x}} M\left(2\pi i {\mathbf{v}},
\frac{d}{dt}\right)\Theta =0\cdot e^{2\pi i \mathbf{v} \cdot
  \mathbf{x}}=0,
$$ 
thanks to the fact that $\Theta \in
     {\mathfrak{B}}_{{\mathcal{D}}'({\mathbb{R}})} (M(2\pi i
     {\mathbf{v}}, \tau))$ giving
$$
M\left(2\pi i {\mathbf{v}}, \frac{d}{dt}\right) \Theta=0.
$$ 
As
${\mathfrak{B}}_{{\mathcal{D}}'_{{\mathbb{A}}}({\mathbb{R}}^{d+1})}(M)$
is approximately controllable in time $\textrm{T}$, there exists a
sequence $(w_k)_{k\in \mN}$ in
${\mathfrak{B}}_{{\mathcal{D}}'_{{\mathbb{A}}}({\mathbb{R}}^{d+1})}(M)$
such that $(w_k-w_2)|_{(\textrm{T},\infty)}$ converges to $0$ in
$(L^\infty((\textrm{T},\infty)\times \mR^d))^n$.  But this implies
that $(w_k-w_2)|_{(\textrm{T},\infty)}$ converges to $0$ in
$({\mathcal{D}}'_{{\mathbb{A}}}((\textrm{T},\infty) \times \mR^d)^n$,
and owing to the continuity of the Fourier transform $\widehat{\cdot}:
{\mathcal{D}}'_{{\mathbb{A}}}({\mathbb{R}}^{d+1})\rightarrow
\calD'(\mR^{d+1} )$ with respect to the spatial variables, it follows
that $ (\widehat{w_k}-\widehat{w_2})|_{(\textrm{T},\infty)}$ converges
to $0$ in $(\calD'((\textrm{T},\infty) \times \mR^d )^n$. We can write
$$ 
\widehat{w_k}=\sum_{\mathbf{v}\in A^{-1}{\mathbb{Z}}^d }
\delta_{\mathbf{v}} \otimes T_k^{(\mathbf{v})} ,
$$ 
where each $T_k^{(\mathbf{v})} \in (\calD'(\mR))^n$. Since
$\widehat{w_k}|_{(\textrm{T},\infty)}$ converges to
$\widehat{w_2}|_{(\textrm{T},\infty)}$ in the space
$(\calD'((\textrm{T},\infty) \times \mR^d ))^n$, it follows that
$T_k^{(\mathbf{v})}|_{(\textrm{T},\infty)}$ converges in $\calD'((\textrm{T},\infty))$ to
$\Theta|_{(\textrm{T},\infty)} $ for all $ \mathbf{v}\in
A^{-1}{\mathbb{Z}}^d$. Also, $T_k^{(\mathbf{v})}|_{(-\infty,0)}=0$.
As
$$
M\left(2\pi i {\mathbf{v}}, \frac{d}{dt}\right) T_k^{(\mathbf{v})}=0,
$$ 
so that $T_k^{(\mathbf{v})}\in
{\mathfrak{B}}_{{\mathcal{D}}'({\mathbb{R}})}(M(2\pi i \bbv , \tau))$.
 
Now suppose that there exists a nontrivial element $[\bbm]$ in the
$\mC[\tau]$-quotient module $ \mC[\tau]^{1\times n}/ \langle M(2\pi i
\bbv , \tau)\rangle$ and a nonzero polynomial $p \in \mC[\tau]$ such
that $p \cdot \bbm \in \langle M(2\pi i \bbv , \tau)\rangle$.  As $
\bbm \not\in \langle M(2\pi i \bbv , \tau)\rangle$, it follows from
the cogenerator property of $\calD'(\mR)$ (see for example
Definition~3.4 on page 774 and the paragraph following the proof of
Lemma~3.5 on page 775 of \cite{Woo}) that $\bbm(d/dt)$ is not
identically $0$ on
${\mathfrak{B}}_{{\mathcal{D}}'({\mathbb{R}})}(M(2\pi i \bbv ,
\tau))$. Let the element $w_0\in
    {\mathfrak{B}}_{{\mathcal{D}}'({\mathbb{R}})}(M(2\pi i \bbv ,
    \tau))$ be such that
$$
\bbm\Big(\frac{d}{dt}\Big)\neq 0.
$$ 
Without loss of generality, we may assume that
$$ 
u_0:=\bbm\Big(\frac{d}{dt}\Big)\bigg|_{(0,\infty)}\neq 0
$$ 
(otherwise $w_0$ can be shifted to achieve this). As all
topological vector spaces are Hausdorff (\cite[Theorem~1.12]{Rud}), it
follows that in the topological vector space $\calD'((0,\infty))$,
there exists a neighbourhood $N$ of $u_0$ that does not contain
$0$. Since the map
$$
\bbm\Big(\frac{d}{dt}\Big): (\calD'((0,\infty)))^n\rightarrow \calD'((0,\infty))
$$ 
is continuous, there exists a neighbourhood $N_1$ of
$w_0|_{(0,\infty)}$ in $(\calD'((0,\infty)))^n$ such that
$\widetilde{w}_0 \in N_1$ implies that
$$
\bbm\Big(\frac{d}{dt}\Big)\widetilde{w}_0 \in N. 
$$ 
Choose $k$ large enough so that $\bbS_{-\textrm{T}}
T_{k}^{(\mathbf{v})} \in N_1$. Set
$$
u:= \bbm\Big(\frac{d}{dt}\Big)T_{k}^{(\mathbf{v})}.
$$
Then we have $u\neq 0$. 
 
On the other hand, since $p\cdot \bbm \in \langle M(2\pi i \bbv ,
\tau)\rangle$ and since the element $T_k^{(\mathbf{v})}\in
    {\mathfrak{B}}_{{\mathcal{D}}'({\mathbb{R}})}(M(2\pi i \bbv ,
    \tau))$, it follows that
$$
p\Big(\frac{d}{dt}\Big) u=0.
$$ 
But we know that since $T_k^{(\mathbf{v})}|_{(-\infty,0)}=0$, also
$$
u|_{(-\infty,0)}= \bbm\Big(\frac{d}{dt}\Big)T_{k}^{(\mathbf{v})}|_{(-\infty,0)}=0.
$$ 
By the autonomy of behaviours corresponding to nonzero polynomials
(see for example \cite[Theorem~1.2]{Sas}) we conclude that $u=0$, a
contradiction to the last sentence in the previous paragraph.  This
completes the proof of (1)$\Rightarrow$(4).
  
\medskip 
  
\noindent {\bf (4) $\Rightarrow$ (3)}: Suppose that (3) does not hold,
and let $\mathbf{v} \in A^{-1} {\mathbb{Z}}^d$ be such that
\begin{equation}
\label{eq_proof_of_4_giving_3}
\neg \bigg( \exists r_{\mathbf{v}}\in \mZ \textrm{ with } 0\leq
r_{\mathbf{v}}\leq \min\{n,m\}\textrm{ and }\forall t\in {\mathbb{C}},
\textrm{ rank}\left( M( 2\pi i \mathbf{v}, t) \right) =
r_{\mathbf{v}}\bigg).
\end{equation}
From \cite[Theorem~B.1.4, page 404]{PolWil}, it follows that there
exist unimodular polynomial matrices $U,V$ with entries from
$\mC[\tau]$ such that
$$
M(2\pi i \mathbf{v}, \tau)= U \Sigma V,
$$
where 
$$ 
\Sigma:=\left[\begin{array}{c|c}\begin{array}{ccc} 
d_1 & & \\ &
\ddots & \\ && d_r
\end{array} & \mathbf{ 0}\\  \hline  \mathbf{0}&\mathbf{ 0}
\end{array} \right],
$$ 
and the $d_k$s polynomials such that $d_k$ divides $d_{k+1}$ for
all $k\in \{1,\cdots, r-1\}$. Thus any vector in $\langle M(2\pi i
\mathbf{v}, \tau) \rangle$ is of the form
\begin{equation}
\label{eq_proof_of_4_giving_3b}
\bbu U \Sigma V = \widetilde{\bbu} \Sigma V
= \widetilde{u_1} d_1 \bbv_1 + \cdots+ \widetilde{u_r} d_r \bbv_r,
\end{equation}
where $\bbu \in \mC[\tau]^{1\times m}$, 
$$ 
\widetilde{\bbu}:=\bbu U= \left[ \begin{array}{ccc} \widetilde{u_1}
    & \cdots & \widetilde{u_r}\end{array}\right],
$$ 
and $\bbv_1,\cdots \bbv_n$ are the rows of $V$. Let
$\bbm:=\bbv_r$. It follows that $d_r$ is not constant thanks to
\eqref{eq_proof_of_4_giving_3}. Clearly $\bbm \not\in \langle M(2\pi i
\mathbf{v}, \tau) \rangle$, for otherwise we would obtain
$\widetilde{u_1} d_1 \bbv_1 + \cdots+ \widetilde{u_r} d_r
\bbv_r=\bbv_r$, and so $\widetilde{u_r} d_r =1$, contradicting the
fact that $d_r$ is not a constant. Moreover, from
\eqref{eq_proof_of_4_giving_3}, $d_r \cdot \bbm \in \langle M(2\pi i
\mathbf{v}, \tau) \rangle$. Hence $[\bbm]$ is a nontrivial torsion
element in the $\mC[\tau]$-module $\mC[\tau]^{1\times n}/\langle
M(2\pi i \bbv, \tau) \rangle$. Consequently, $\mC[\tau]^{1\times
  n}/\langle M(2\pi i \bbv, \tau) \rangle$ is not torsion free, that
is, (4) does not hold. Hence we have shown that $\neg$(3)$\Rightarrow
\neg$(4), that is, (4)$\Rightarrow $(3).

\medskip 

\noindent {\bf (3) $\Rightarrow$ (2)}: Suppose that
$\textrm{T}>0$. Let $w_1, w_2 \in
          {\mathfrak{B}}_{{\mathcal{D}}'_{{\mathbb{A}}}({\mathbb{R}}^{d+1})}(M)$. Then
          we have
$$ 
\widehat{w_1} = \sum_{\mathbf{v} \in A^{-1} {\mathbb{Z}}^d}
          \delta_{\mathbf{v}}\otimes T_{1}^{(\mathbf{v})}, \quad
          \widehat{w_2} = \sum_{\mathbf{v} \in A^{-1} {\mathbb{Z}}^d}
          \delta_{\mathbf{v}}\otimes T_{2}^{(\mathbf{v})},
$$ 
for some $T_1^{(\mathbf{v})}, T_2^{(\mathbf{v})} \in
          ({\mathcal{D}}'({\mathbb{R}}))^n$.  Moreover, owing to the
          correspondence between
          ${\mathcal{D}}'_{{\mathbb{A}}}({\mathbb{R}}^d)$ and the
          space of sequences $s'({\mathbb{Z}}^d)$ of at most
          polynomial growth, it follows that for each $\varphi \in
          {\mathcal{D}}({\mathbb{R}})$, there exist $M_\varphi >0$ and
          a positive integer $k_\varphi$, such that we have the
          estimates
$$
\|\langle T_{1}^{(\mathbf{v})}, \varphi\rangle\|_2  \leq  M_{\varphi}(1+\|\mathbf{n}\|_2)^{k_{\varphi}} ,\quad 
\|\langle T_{2}^{(\mathbf{v})} , \varphi\rangle\|_2  \leq  M_{\varphi}(1+\|\mathbf{n}\|_2)^{k_{\varphi}} ,
$$ 
for all $\mathbf{n}:=A\mathbf{v} \in {\mathbb{Z}}^d$, and where
$\|\cdot\|_2$ is the usual Euclidean norm.  Let  $\theta \in
C^\infty({\mathbb{R}})$ be such that $\theta(t)=1$ for all $t\leq 0$,
$\theta(t)=0$ for all $t>\mathrm{T}/4$ and $0\leq \theta (t)\leq 1$
for all $t\in {\mathbb{R}}$. Define $T^{(\mathbf{v})}\in
({\mathcal{D}}'({\mathbb{R}}))^n$ by
$$
T^{(\mathbf{v})}:= \theta T_1^{(\mathbf{v})}+\theta(\mathrm{T}-\cdot) T_2^{(\mathbf{v})}.
$$ 
Set $\widehat{w}\in {\mathcal{D}}'({\mathbb{R}}^{d+1})$ to be 
$$
\widehat{w}=\displaystyle \sum_{\mathbf{v} \in A^{-1}
  {\mathbb{Z}}^d}\delta_{\mathbf{v}}\otimes T^{(\mathbf{v})}.  
  $$ 
  Then
for every $\varphi \in {\mathcal{D}}({\mathbb{R}})$, we have
\begin{eqnarray*}
\|\langle T^{(\mathbf{v})}, \varphi\rangle\|_2  &\leq& \|\langle \theta T_1^{(\mathbf{v})}, \varphi \rangle \|_2+ 
\|\langle \theta(\mathrm{T}-\cdot) T_2^{(\mathbf{v})}, \varphi \rangle \|_2 \\
&\leq & M_{\theta \varphi} (1+\|\mathbf{n}\|_2)^{k_{\theta \varphi}}+ 
M_{\theta (\mathrm{T}-\cdot)\varphi} (1+\|\mathbf{n}\|_2)^{k_{\theta(\mathrm{T}-\cdot) \varphi}}\\
&\leq & \max\{M_{\theta \varphi},M_{\theta (\mathrm{T}-\cdot)\varphi}\}  
(1+\|\mathbf{n}\|_2)^{\max\{k_{\theta \varphi},k_{\theta(\mathrm{T}-\cdot) \varphi} \}},
\end{eqnarray*}
and so $\vec{w}(\varphi)\in
{\mathcal{D}}'_{{\mathbb{A}}}({\mathbb{R}}^d)$.  Thus $w\in
{\mathcal{D}}'_{{\mathbb{A}}}({\mathbb{R}}^{d+1})$.  Also, $w\in
{\mathfrak{B}}_{{\mathcal{D}}'_{{\mathbb{A}}}({\mathbb{R}}^{d+1})}(M)$
because
$$
M\left( 2\pi i \mathbf{y}, \frac{\partial}{\partial t} \right) ( \delta_{\mathbf{v}}\otimes T^{(\mathbf{v})} ) 
= 
M\left( 2\pi i \mathbf{v}, \frac{d}{d t} \right)( \delta_{\mathbf{v}}\otimes T^{(\mathbf{v})} ) 
=
0,
$$ 
for each $\mathbf{v} \in A^{-1}{\mathbb{Z}}^d$, and so
$$
M\left( 2\pi i \mathbf{y}, \frac{\partial}{\partial t} \right) \widehat{w} =0.
$$ 
Consequently, 
$$ 
M\left( \frac{\partial}{\partial x_1},\cdots,
\frac{\partial}{\partial x_d} , \frac{\partial}{\partial t} \right) w
=0,
$$
that is, $w\in {\mathfrak{B}}_{{\mathcal{D}}'_{{\mathbb{A}}}({\mathbb{R}}^{d+1})}(M)$.

Finally, because $
T^{(\mathbf{v})}|_{(-\infty,0)}=T_1^{(\mathbf{v})}|_{(-\infty,0)}$ and
$ T^{(\mathbf{v})}|_{(\mathrm{T},
  +\infty)}=T_2^{(\mathbf{v})}|_{(\mathrm{T}, +\infty)}$, it follows
that $\widehat{w}|_{(-\infty, 0)}=\widehat{w_1}|_{(-\infty,0)}$ and
$\widehat{w}|_{(\mathrm{T}, +\infty)} =\widehat{w_2}|_{(\mathrm{T},
  +\infty)}$. Consequently, $w|_{(-\infty, 0)}=w_1|_{(-\infty,0)}$ and
$w|_{(\mathrm{T}, +\infty)}=w_2|_{(\mathrm{T}, +\infty)}$, showing
that the behaviour is controllable in time $\mathrm{T}$. This
completes the proof of (3)$\Rightarrow $(2).

\medskip 

\noindent \noindent {\bf (2) $\Rightarrow$ (1)}: This follows trivially
from the definitions.
\end{proof}

\end{document}